\documentclass[12pt]{article}
\usepackage[colorlinks,
            linkcolor=blue,
            anchorcolor=green,
            citecolor=blue
            ]{hyperref}
\usepackage{mathrsfs,amssymb,amsfonts}
\usepackage{amsmath,amssymb,latexsym,color}
\usepackage[mathscr]{eucal}
\usepackage[numbers,sort&compress]{natbib}
\usepackage{graphicx}
\usepackage{lipsum}
\def\thefootnote{\fnsymbol{footnote}}
\newtheorem{thm}{Theorem}[section]

\newtheorem{lemma}[thm]{Lemma}

\newtheorem{example}[thm]{Example}

\newtheorem{problem}[thm]{Problem}
\newtheorem{remark}[thm]{Remark}
\newtheorem{Notation}[thm]{Notation}

\newtheorem{con}[thm]{Conjecture}

\newcommand{\proof}{{\it Proof.\quad}}
\newcommand{\qed}{\hfill\Box\medskip}
\renewcommand{\thefootnote}{\arabic{footnote}}
\newcommand{\di}{{\rm d}}
\newcommand{\dm}{{\rm dim_m}}
\usepackage{CJK}
\begin{document}
\begin{CJK*}{GBK}{song}
\renewcommand{\abovewithdelims}[2]{
\genfrac{[}{]}{0pt}{}{#1}{#2}}

\title{\bf Graphs attaining an upper bound on the mixed metric dimension}

\author{{Shi Chen and Xuanlong Ma\footnote{Corresponding author}}
\medskip
\\
{\small\em School of Science, Xi'an Petroleum University, Xi'an 710065, P. R. China}
\medskip
}

 \date{}
 \maketitle
\newcommand\blfootnote[1]{%
\begingroup
\renewcommand\thefootnote{}\footnote{#1}%
\addtocounter{footnote}{-1}%
\endgroup
}
\begin{abstract}
Given a graph $G$, we show that the mixed metric dimension of $G$ is exactly $\ell(G)+2c(G)$
if and only if $G$ is either a cactus graph in which every cycle has precisely one vertex of degree at least $3$, or a
balanced $\Theta$-graph, where $\ell(G)$ and $c(G)$ denote  the number of leaves and
the cyclomatic number of $G$, respectively.
This provides an affirmative answer to a conjecture proposed
by Sedlar and \v{S}krekovski (2021).


\end{abstract}

\blfootnote{{\em E-mail addresses:} nuoyoushi@gmail.com (Chen), xuanlma@xsyu.edu.cn (Ma)}

\section{Introduction}\label{pre}
All graphs considered in this paper are finite, simple, and connected.
We always use $G$ to denote a graph.
Denote by $V(G)$ and $E(G)$ the vertex set and edge set of $G$, respectively.
The {\em minimum degree} of a graph $G$, denoted by $\delta(G)$, is the smallest integer $k$ such that there exists a vertex with degree $k$ in this graph.
For two vertices $u$ and $v$ of a graph $G$, the distance between
$u$ and $v$ is denoted by $\di_G(u,v)$, or simply by $\di(u,v)$ when no confusion can arise.
For distinct $x,y\in V(G)$, we say a vertex $v\in V(G)$
{\em distinguishes} $x$ and $y$ whenever
$
\di(v,x)\ne \di(v,y).
$
A subset $S\subseteq V(G)$ is called a {\em metric generator} of $G$ if for every
pair of distinct vertices $x,y\in V(G)$, there exists a vertex in $S$ such that this vertex distinguishes $x$ and $y$.
A metric generator of $G$ with minimum cardinality is called
a {\em metric basis}
of $G$, and the cardinality of a metric basis is referred to as the {\em metric dimension} of $G$.

In the 1970s,
the concept of metric dimension was introduced independently by Harary and Melter \cite{Ham}, and Slater \cite{Sl}.
Their aim is to distinguish vertices of a graph using distances.
The metric dimension of a graph has a wide range of applications, including robot navigation in networks, chemistry, pattern recognition, and image processing, and so on. See the survey \cite{Ti} for more results and applications of the metric dimension.
In 2018,
as a further variation of the classical notion of a metric generator, Kelenc et al. \cite{KT} introduced
the concept of an edge metric generator, which is
to distinguish the edges of a graph by means of distances from a prescribed set of vertices.
For a vertex $v\in V(G)$ and an edge $e=\{u,w\}\in E(G)$, the distance between $v$ and $e$ is defined by
$$
\di(v,e)=\min\{\di(v,u),\di(v,w)\}.
$$
A vertex $x\in V(G)$ is said to distinguish two distinct edges $e_1,e_2\in E(G)$ if
$
\di(x,e_1)\ne \di(x,e_2).
$
A set $S\subseteq V(G)$ is called an {\em edge metric generator} for $G$ provided that every pair of distinct edges of $G$ is distinguished by some vertex in $S$.
An edge metric generator of minimum cardinality is called an {\em edge metric basis} of $G$, and its cardinality is called the {\em edge metric dimension} of $G$.

For a graph $G$, let
$$
O(G)=V(G)\cup E(G).
$$
We say a vertex $v$ of $(G)$ distinguishes two distinct elements $x,y\in O(G)$ (vertices or edges) if
$
\di(v,x)\ne \di(v,y).
$
For a subset $S$ of $V(G)$, we say $S$ is a
{\em mixed metric generator} of $G$ provided that, for every pair of distinct elements $x,y\in O(G)$, there exists a vertex in $S$ such that this vertex distinguishes $x$ and $y$.
The cardinality of a mixed metric generator with minimum size, denoted by $\dm(G)$, is called the {\em mixed metric dimension} of $G$.
Kelenc et al. \cite{Kek} first introduced the concept of a mixed metric dimension.

In \cite{Kek},
the authors determined the mixed metric dimension of several families of graphs and established an upper bound in terms of the girth. They also proved that computing the mixed metric dimension of a general graph is NP-hard.
A {\em leaf} of $G$ is a vertex of degree $1$.
Denote by $\ell(G)$
the number of all leaves of $G$.
The {\em cyclomatic number} of $G$, denoted by $c(G)$, is
the number of independent cycles of $G$. It is well known that $c(G)=|E(G)|-|V(G)|+1$.
In \cite{Sed21},
Sedlar and \v{S}krekovski obtained the exact value of the mixed metric dimension of a unicyclic graph and they proposed the following conjecture:

\begin{con}\label{conj}{\rm (\cite[Conjecture 13]{Sed21})}
Let $G$ be a connected graph which is not a cycle.
Then $\dm(G)\le \ell(G)+2c(G)$.
\end{con}

In \cite{Sed21}, the authors showed that
Conjecture~\ref{conj} holds for trees and $3$-connected
graphs.
A {\em Theta graph}
(or $\Theta$-graph) is a graph with exactly two vertices of degree $3$ and all others of degree $2$.
Let $H$ be a $\Theta$-graph with distinct vertices $u,v$ of degree $3$.
Then $H$ is called {\em balanced} provided that the lengths of all three paths connecting $u$ and $v$ differ by at most $1$; Otherwise, $H$ is called {\em unbalanced}.
In \cite{Sed}, the authors proved that Conjecture~\ref{conj} holds for all $\Theta$-graph. In 2025,
Chakraborty et al. \cite{Ch} proved that Conjecture~\ref{conj} is valid for a connected graph which is not a tree such that $\ell(G)\ge 1$, and a connected graph $G$ with $\ell(G)=0$ and has a cut-vertex.
In \cite{Hua}, the authors gave some graphs satisfying Conjecture~\ref{conj}. Recently, Chen and Ma \cite{C} showed  that for a $2$-connected graph $G$ which is not a cycle, the mixed metric dimension of $G$ is at most $2c(G)$.
As an immediate application, Conjecture~\ref{conj} was proved.

If $G$ is a cycle, then by
\cite[Theorem~3.8 and Proposition~4.1]{Kek} we have $\dm(G)=3$. As a result, $\dm(G)=\ell(G)+2c(G)$ is not valid for a cycle. Moreover,
by \cite[Corollary 8]{Sed},
$\dm(G)=\ell(G)+2c(G)$ is valid for a
balanced $\Theta$-graph $G$.
In the concluding section of \cite{Sed}, the authors proposed the following conjecture:

\begin{con}\label{conj-2}{\rm (\cite[Conjecture 12]{Sed})}
Let $G$ be a graph. Then
\begin{equation}\label{th-eq-1}
\dm(G)=\ell(G)+2c(G)
\end{equation}
if and only if $G$ is either a cactus graph in which every cycle has
precisely one vertex of degree at least $3$ or a
balanced $\Theta$-graph.
\end{con}

This paper will provide an affirmative answer to Conjecture~\ref{conj-2}.

\section{Preliminaries}

In order to prove Conjecture~\ref{conj-2}, we will establish several lemmas in this section.

\begin{lemma}\label{lem-1}
Let $G$ be a $2$-connected graph with $c(G)=2$. Then $G$ is a $\Theta$-graph.
\end{lemma}
\proof
Note that $G$ is not a cycle and $\delta(G)\ge 2$. Then we have
$$
\sum_{v\in V(G)}(\deg(v)-2)=2|E(G)|-2|V(G)|=2c(G)-2=2.
$$
We conclude that the degree of a vertex of $G$ is at most $4$.

Suppose, for a contradiction, that $G$ has a vertex $u$ with degree $4$. Then all other vertices of $G$ have  degree $2$. Let $x,y,z,w$ be the four vertices adjacent to $u$. If we delete $u$ in $G$, then any of $x,y,z,w$ has degree $1$.
It follows that the subgraph of $G$ obtained by delete $u$ is the union of two disjoint paths. Thus, $u$ is a cut-vertex of $G$, this contradicts that $G$ is $2$-connected.

Thus, $G$ has precisely two vertices with degree $3$, and all other vertices have degree $2$. This means that
$G$ is a $\Theta$-graph, as desired.
$\qed$

\begin{lemma}\label{lem-2}
Let $G$ be a balanced $\Theta$-graph and let $a,b$ be two distinct vertices of $G$.
Then there exists a mixed metric generator $\mathcal{R}$ of $G$
such that
$$
|\mathcal{R}|=4,~~a\in \mathcal{R},~~\di(r,a)\ge \di(r,b) \text{ for some $r\in \mathcal{R}$}.
$$
\end{lemma}
\proof
Let $x,y$ be the two vertices of degree $3$, and let these three paths connecting $x$ and $y$ be
$$
P_i=(v_{i,0},v_{i,1},v_{i,2},\ldots,v_{i,l_i}),~~
x=v_{i,0},~~y=v_{i,l_i},
$$
where $1\le i \le 3$. Let $q=\min\{l_1,l_2,l_3\}$. Then
$l_i\in \{q,q+1\}$ for any $1\le i \le 3$.
Note that among the three paths $P_1,P_2,P_3$, there are at most two paths of length $q+1$.

\medskip
\noindent {\bf Claim A.}
Take $u \in V(P_k)\setminus\{x,y\}$ and $v \in V(P_l)\setminus\{x,y\}$ from two distinct paths $P_k, P_l$, and if the path $P_t$ has length $q+1$, then one of $u,v$ must lie on $P_t$. Then $\{x,y,u,v\}$ is a mixed metric generator of $G$.
\medskip

We refer to the above $u,v$ as marked vertices.
In the following, denote by $[t,t]_i$ and $[t,t+1]_i$ the vertex $v_{i,t}$ and edge $\{v_{i,t},v_{i,t+1}\}$ in $P_i$, respectively. Let $\varepsilon\in \{0,1\}$. Then for any $[t,t+\varepsilon]_i\in O(P_i)\setminus\{x,y\}$, we have that
\begin{equation}\label{eq-1}
\di(x,[t,t+\varepsilon]_i)=t,~~
\di(y,[t,t+\varepsilon]_i)=l_i-t-\varepsilon.
\end{equation}
By \eqref{eq-1}, if $P_i$ has length $q$, then
any pair of distinct elements of $O(P_i)$ can be  distinguished by one of $x,y$.
Suppose that $P_i$ has length $q+1$. Let $u$ be the marked vertex of $P_i$. Then $u$ can distinguish not only $x$ and $[0,1]_i$, but also $y$ and $[l_i-1,l_i]_i$. By \eqref{eq-1} again,
every two distinct elements of $O(P_i)$ can be distinguished by one of $x,y,u$.

In the following, we consider two elements $z,z'$ lying on distinct paths. We may assume that every of $x,y$ can not distinguish $z,z'$. In view of \eqref{eq-1}, let
$$
z=[t,t+\varepsilon]_i,~~z'=[t,t+\varepsilon']_j,~~
i\ne j,
$$
where $\varepsilon'\in \{0,1\}$. Since $\di(y,z)=\di(y,z')$, we have that
\begin{equation}\label{eq-2}
l_i-\varepsilon=l_j-\varepsilon'.
\end{equation}
By the choice of a marked vertex, at least one marked vertex comes from $P_i$ or $P_j$. Without loss of generality, let $u=v_{i,p}\in O(P_i)$ be a  marked vertex. We deduce that
$$
\di(u,z)=h=\min\{|p-t|,|p-t-\varepsilon|\}
$$
and
$$
\di(u,z')=\min\{p+t,l_i-p+l_j-t-\varepsilon'\}.
$$
If $p\le t$, then $h=t-p$, which implies that
$$
(p+t)-h = 2p>0,~~(l_i-p+l_j-t-\varepsilon')-h
=2l_i-2t-\varepsilon>0~(\text{by \eqref{eq-2}}).
$$
If $p\ge t+1$, then $h=p-t-\varepsilon$, this forces that
$$
(p+t)-h = 2t+\varepsilon>0,~~(l_i-p+l_j-t-\varepsilon')-h
=2(l_i-p)>0~(\text{by \eqref{eq-2}}).
$$
Combining the above two cases, we conclude that
$\di(u,z')>\di(u,z)$. It follows that $u$ distinguishes $z,z'$. As a result, the set of $x,y$ and all marked vertices is a mixed metric generator of $G$, which proves Claim A.

Now, suppose that $a=x$.
Take $\mathcal{R}=\{x,y,u,v\}$ as in Claim A.
Then for any $b\in V(G)\setminus\{a\}$,
we infer that $\di(y,a)\ge \di(y,b)$, which concludes the proof of this lemma. Similarly, we also conclude our proof if $a=y$. Thus, in the following, let $a\notin \{x,y\}$.

Suppose that $a$ can serve as a marked vertex. Let $u$ be the other marked vertex. Then by Claim A,
$\{x,y,a,u\}$ is a mixed metric generator of $G$.
Now for any $v\in V(G)$, it follows that
\begin{equation}\label{eq-3}
\di(x,v)+\di(y,v)\in \{q,q+1\}.
\end{equation}
If both $\di(x,a)<\di(x,b)$ and $\di(y,a)<\di(y,b)$ hold, then
$$\di(x,b)+\di(y,b)\ge \di(x,a)+\di(y,a)+2,$$
which is impossible as \eqref{eq-3}.
We conclude that either $\di(x,a)\ge \di(x,b)$ or $\di(y,a)\ge \di(y,b)$, which proves our lemma.

Suppose that $a$ cannot serve as a marked vertex. Then it must be that among the three paths $P_1,P_2,P_3$, there are precisely two paths of length $q+1$, say $P_2$ and $P_3$, and
$a$ is an internal vertex of $P_1$. For convenience, now let
$$
P_1=(a_{0},a_{1},\ldots,a_{q}),~~
P_2=(b_{0},b_{1},\ldots,b_{q+1}),~~
P_3=(c_{0},c_{1},\ldots,c_{q+1}),
$$
where $x=a_{0}=b_{0}=c_{0}$ and $y=a_{q}=b_{q+1}=c_{q+1}$. Write $a=a_t$ for some $1\le t \le q-1$.

\medskip
\noindent {\bf Claim B.} $\{x,a_t,b_1,c_q\}$ is a mixed metric generator of $G$.
\medskip

For any two elements $z,z'$ of $O(G)$,
if $\di(x,z)\ne \di(x,z')$, then $x$ distinguishes them.
As a result, we next consider the elements at distance
$i$ from $x$.
For $1\le i \le q-1$,
we obtain Table~\ref{tab-1},
where
$\alpha_i=\min\{|t-i|,|t-i-1|\}$, $\mu_i=\min\{t+i,2q-t-i+1\}$, and
$\nu_i=\min\{t+i,2q-t-i\}$.
\begin{table}[h]\label{tab-1}
\centering
\caption{Distance table for $1\le i \le q-1$}
\medskip
\begin{tabular}{cccc}
\hline
$\di(\cdot,\cdot)$ & $a_t$ & $b_1$ & $c_q$ \\
\hline
$a_i$ &       $|t-i|$ & $i+1$ & $q-i+1$ \\
$\{a_i,a_{i+1}\}$ & $\alpha_i$ & $i+1$ & $q-i$ \\
$b_i$ &       $\mu_i$ & $i-1$ & $q-i+2$ \\
$\{b_i,b_{i+1}\}$ & $\nu_i$ & $i-1$ & $q-i+1$ \\
$c_i$ &       $\mu_i$ & $i+1$ & $q-i$ \\
$\{c_i,c_{i+1}\}$ & $\nu_i$ & $i+1$ & $q-i-1$ \\
\hline
\end{tabular}
\end{table}
Now by Table~\ref{tab-1}, with the exception of the  pair $c_i, \{a_i,a_{i+1}\}$, every pair of elements of $\{a_i,b_i,c_i,\{a_i,a_{i+1}\},\{b_i,b_{i+1}\},
\{c_i,c_{i+1}\}\}$ is distinguished by one element of $b_1,c_q$. If $t\le i$, then
$$
\alpha_i=i-t,~~t+i-\alpha_i=2t, ~~ 2q-t-i+1-\alpha_i=2q-2i+1>0.
$$
If $t\ge i+1$, then
$$
\alpha_i=t-i-1,~~t+i-\alpha_i=2i+1, ~~ 2q-t-i+1-\alpha_i=2q-2t+2>0.
$$
It follows that $\alpha_i<\mu_i$, and so
$c_i$ and $\{a_i,a_{i+1}\}$ can be distinguished by $a_t$.

Now, we consider the elements at distance
$0$ from $x$. There are only four such elements   $x,\{x,a_1\},\{x,b_1\},\{x,c_1\}$.
Note that the distance vectors from
the four elements $x,\{x,a_1\},\{x,b_1\},\{x,c_1\}$ to $(a_t,b_1,c_q)$ are respectively
$$
(t,1,q),~~(t-1,1,q),~~(t,0,q),~~(t,1,q-1).
$$
It follows that every two of $x,\{x,a_1\},\{x,b_1\},\{x,c_1\}$ can be distinguished by some element in $\{a_t,b_1,c_q\}$.

Finally, we consider the elements at distance
$q$ from $x$. They are
\begin{equation}\label{eq-4}
y,~~b_q,~~c_q,~~\{b_q,y\},~~\{c_q,y\}.
\end{equation}
Now the distance vectors from
the five elements in \eqref{eq-4} to $(a_t,b_1,c_q)$ are
$$
(q-t,q,1),(q-t+1,q-1,2),(q-t+1,q+1,0),
(q-t,q-1,1),(q-t,q,0),
$$
respectively.
It follows that every two in \eqref{eq-4} can be distinguished by some element in $\{a_t,b_1,c_q\}$.
Based on the above proof, Claim B holds.

If $b=x$, then taking $r=x$, we deduce that
$\di(a,r)=t\ge 0=\di(b,r)$, as desired.
If $b=y$, then taking $r=c_q$, we conclude that
$\di(a,r)=q-t+1\ge 1=\di(b,r)$, as required.
In the following, we may assume that $b\notin \{x,y\}$.

Suppose, for a contradiction,
that for any $r\in \{x,b_1,c_q\}$, we all have
$\di(a,r)<\di(b,r)$. If $b=a_i$ for $1\le i \le q-1$, then for $r=x$ and $c_q$, we have that $t<i$ and $q-t+1<q-i+1$ respectively, a contradiction.
If $b=b_j$ for $1\le j \le q$, then for $r=x$ and $c_q$, we deduce that $t<j$ and $j<t+1$ respectively, which is impossible.
Finally, if $b=c_k$ for $1\le k \le q$, then for $r=x$ and $c_q$, we infer that $t<k$ and $k<t-1$ respectively, which is impossible.
$\qed$

Denote by $L(G)$ the set of all leaves of $G$.
A {\em block} of a graph is a maximal connected subgraph that has no cut-vertex.
Given a connected graph $G$ with at least a cut-vertex, we next define the graph $B(G)$ with vertex set $\mathcal{B}\cup C$, where
$\mathcal{B}$ is the set of all blocks of $G$, and $C$ is the set of all cut-vertices of $G$.
In fact, $B(G)$ is bipartite with bipartition $(\mathcal{B},C)$, and a block $B\in \mathcal{B}$ and a cut-vertex $v\in C$ is adjacent if and only
if $v\in V(B)$. Clearly, $B(G)$ is a tree, and so we call $B(G)$ the {\em block tree} of $G$
(cf. \cite{Jon18}).

For a connected graph $G$ with at least a cut-vertex,
let $\mathcal{F}$ be the set of all blocks containing a cycle.
For $D\in \mathcal{F}$, let $X_D$ be a mixed metric generator of $D$, and let $C_D$ be the set of the cut-vertices in $D$.

\begin{lemma}\label{lem-3}
Let $G$ be a connected graph with at least a cut-vertex. Then
$$
L(G)\cup(\bigcup_{D\in \mathcal{F}}X_D\setminus C_D)
$$
is a mixed metric generator of $G$.
\end{lemma}
\proof
If every $B\in \mathcal{B}$ has no cycle, then $G$ is a tree, and so $L(G)$ is a mixed metric generator of $G$, as desired.
In the following, assume that $\mathcal{F}\ne \emptyset$ and let $D\in \mathcal{F}$. Now let
$c$ be a cut-vertex belonging to $D$, and let $E\in \mathcal{B}\setminus \{D\}$ with $c\in E$.

\medskip
\noindent {\bf Claim A.} Suppose that $E$ belongs to $\mathcal{F}$. Then $(X_D\cup X_E)\setminus\{c\}$ is a mixed metric generator of $D\cup E$.
\medskip

If two distinct elements of $O(D)$ (resp. $O(E)$) are distinguished by $c$, then they must be distinguished by an element of $X_E\setminus\{c\}$ (resp. $X_D\setminus\{c\}$).
Now take
$z_1\in O(D)\setminus\{c\},z_2\in O(E)\setminus\{c\}.$
Without loss of generality, let
$\di_{D\cup E}(z_1,c)\le \di_{D\cup E}(z_2,c)$.

Suppose, for a contradiction, that for any $x\in X_D\setminus\{c\}$, we have that
\begin{equation}\label{eq-5}
\di_D(x,z_1)=\di_D(x,c)+\di_D(c,z_1).
\end{equation}
If $\di(c,z_1)=0$, then $z_1$ is an edge with an endpoint $c$, which implies that $c$ and $z_1$ cannot
be distinguished by any of $X_D\cup \{c\}$ by \eqref{eq-5}, a contradiction.
Now let $\di_D(c,z_1)=r>0$.
Choose a shortest path form $c$ to $v_{r}$ in $D$ as follows:
$$
(c,v_1,v_2,\ldots,v_{r-1},v_r),
$$
where if $z_1$ is a vertex, then $v_r=z_1$; otherwise,  $v_r$ is an endpoint of $z_1$. Then
\begin{equation}\label{eq-6}
\di_D(x,v_{r-1})=\di_D(x,c)+r-1=
\di_D(x,c)+\di_D(c,\{v_{r-1},v_r\})=
\di_D(x,\{v_{r-1},v_r\})
\end{equation}
and
\begin{equation}\label{eq-7}
\di_D(c,v_{r-1})=\di_D(c,\{v_{r-1},v_r\}).
\end{equation}
Naturally, if $v_{r-1}=c$, then both \eqref{eq-6} and \eqref{eq-7} hold.
It follows that, in this case, $\{v_{r-1},v_r\}$ and $v_{r-1}$ cannot
be distinguished by any of $X_D\cup \{c\}$ by \eqref{eq-5}, this contradicts that $X_D$ is a mixed metric generator of $D$.

We conclude that there exists $y\in X_D\setminus\{c\}$, such that
$
\di_D(y,z_1)\le \di_D(x,c)+\di_D(c,z_1).
$
It follows that
$$
\di_{D\cup E}(y,z_1)<\di_{D\cup E}(y,c)+\di_{D\cup E}(c,z_1)\le \di_{D\cup E}(y,c)+\di_{D\cup E}(c,z_2)
=\di_{D\cup E}(y,z_2),
$$
which means that $z_1$ and $z_2$ are distinguished by $y$. As a result, $(X_D\cup X_E)\setminus\{c\}$ is a mixed metric generator of $D\cup E$, so Claim A holds.

\medskip
\noindent {\bf Claim B.} Suppose that $E$ is an
edge $\{c,c'\}$ and $c'$ is a leaf. Then $(X_D\cup \{c'\})\setminus\{c\}$ is a mixed metric generator of $D\cup E$.
\medskip

Claim B easily follows from the proof of Claim A.

Now, suppose that $E$ is a bridge, say $\{c,c_1\}$, where $c_1$ is not a leaf. Then there exists $E_1\in \mathcal{B}$ such that $c_1\in V(E_1)$. Certainly, $E_1$ can also be a bridge, say $\{c_1,c_2\}$, where $c_2$ is not a leaf. Now let
every edge of this path $P=(c,c_1,c_2,\ldots,c_n)$ is a bridge, and let $E_n\in \mathcal{B}$ be either an
edge $\{c_n,c'\}$ and $c'$ is a leaf, or a block containing $c_n$ and belonging to $\mathcal{F}$.
If $E_n\in \mathcal{F}$, then it is similar to
Claim A, $(X_D\cup X_{E_n})\setminus\{c,c_n\}$ is a mixed metric generator of $D\cup E \cup P$.
If not, it is similar to
Claim B, $(X_D\cup \{c'\})\setminus\{c\}$ is a mixed metric generator of $D\cup E \cup P$.

Finally, by the block tree of $G$, we can ``glue" all the blocks of $G$ together at their cut vertices. As a result, $
L(G)\cup(\cup_{D\in \mathcal{F}}X_D\setminus C_D)
$
is a mixed metric generator of $G$, as desired.
$\qed$

Given a graph $G$, let
$$
\mathcal{S}=\{v\in V(G): \deg(v)\ge 3\}.
$$
A {\em thread} in a graph $G$ is a path whose internal vertices all have degree $2$ in $G$.
Let $P=(p_0,p_1,\ldots,p_l)$ be a thread of $G$.
We use $G-P^0$ to denote the subgraph of $G$ that is obtained by removing the internal vertices of this thread $P$ and keeping the endpoints $p_0$ and $p_l$. If $G-P^0$ is connected and $l<\di_{G-P^0}(p_0,p_l)$, then
$P$ is called a {\em short thread} of $G$. If $\deg(p_0)\ge 3$ and $\deg(p_l)\ge 3$, then this thread $P$ is called
{\em maximal}.

\begin{lemma}\label{lem-4}
Let $G$ be $2$-connected with $c(G)\ge 3$. Then $\dm(G)\le 2c(G)-1$.
\end{lemma}
\proof
Let $c(G)=n$.
If every thread of $G$ is short, it follows from \cite[Lemmas~2.3 and 2.5]{C} that
$\dm(G)\le |\mathcal{S}|\le 2n-2<2n-1$, as desired.
In the following, we always assume that $G$ has a maximal thread $P=(u=p_0,p_1,\ldots,p_l=v)$ such that $l\ge \di_H(u,v)$, where $H=G-P^0$. Then clearly, $H$ is connected and has no leaves. Thus, we have that
$$
c(H)=|E(H)|-|V(H)|+1=(|E(G)|-l)-(|V(G)|-(l-1))+1,
$$
namely,
\begin{equation}\label{eq-8}
n=c(G)=c(H)+1.
\end{equation}
To end our proof, we now proceed by induction on $n$.

Suppose first that $n=3$. Then $c(H)=2$
by \eqref{eq-8}.
Assume that $H$ is a $2$-connected graph. Then Lemma~\ref{lem-1} implies that $H$ is a $\Theta$-graph.
If $H$ is an unbalanced $\Theta$-graph, then
in view of \cite[Theorem 10]{Sed}, we obtain that
$\dm(H)=3$, which implies that
$\dm(G)\le \dm(H)+2\le 2n-1$ by \cite[Lemma~2.4]{C}, as required.
Next, assume that $H$ is a balanced $\Theta$-graph.
Then, it follows from Lemma~\ref{lem-2} that
there exists a mixed metric generator $\mathcal{R}$ of $H$
such that
$$
|\mathcal{R}|=4,~~u\in \mathcal{R},~~\di_H(r,u)\ge \di_H(r,v) \text{ for some $r\in \mathcal{R}$}.
$$
Note that $\dm(H)=4$ by \cite[Theorem 7]{Sed}.
Thus, $\mathcal{R}$ is a mixed metric generator of $H$ with the smallest cardinality. Now let
$w=p_{\lceil l/2\rceil}$. In view of \cite[Lemma~2.4]{C}, it is obvious that
$\mathcal{R}\cup \{u,w\}=\mathcal{R}\cup \{w\}$ is a mixed metric generator of $G$. It follows that
$\dm(G)\le 2n-1$, as desired.

Assume that $H$ is not $2$-connected.
Note that $H$ has no leaves and contains at least a cut-vertex. It follows from $c(H)=2$ that $H$ is a
cactus graph with exactly two cycles. Now let $C_u$ and $C_v$ be two distinct cycles with $u\in C_u$ and $v\in C_v$, and let $a_u$ and $a_v$ be cut-vertices with $a_u\in C_u$ and $a_v\in C_v$. Clearly, $u\ne a_u$ and $v\ne a_v$, but $a_u$ and $a_v$
are not necessarily distinct. For $C_u$,
taking a thread from $u$ to $a_u$ satisfying the hypotheses
of \cite[Lemma~2.4]{C}, we deduce that there exists $x\in V(C_u)\setminus\{u,a_u\}$ such that $\{u,a_u,x\}$ is a mixed metric generator of $C_u$. Similarly, we may assume that $\{v,a_v,y\}$ is a mixed metric generator of $C_v$.
It follows from Lemma~\ref{lem-3} that
$\{u,v,x,y\}$ is a mixed metric generator of $H$.
Write, now, $w=p_{\lceil l/2\rceil}$.
By \cite[Lemma~2.4]{C} again, we conclude that
$\{u,v,x,y\}\cup \{w\}$ is a mixed metric generator of $G$.
As a consequence, $\dm(G)\le 2n-1$, as desired.

Now, assume that the statement holds for
every $2$-connected graph whose cyclomatic number is at least $4$ and less than $n$.
We then prove this result holds for $n$.

\medskip
\noindent {\bf Case 1.} $H$ is $2$-connected.
\medskip

If $c(H)=2$, then $c(G)=3$ by \eqref{eq-8}, which implies that the required result holds by the above proof.
Suppose nest that $c(H)\ge 3$.
Note that $c(H)<n$.
Then by the induction hypothesis, we deduce that
$\dm(H)\le 2c(H)-1$. It follows from \cite[Lemma~2.4]{C} that
$\dm(G)\le \dm(H)+2$, which implies that
$$\dm(G)\le \dm(H)+2 \le 2c(H)-1+2=2c(H)+1=2n-1$$
by \eqref{eq-8}, as desired.

\medskip
\noindent {\bf Case 2.} $H$ is not $2$-connected.
\medskip

Then $H$ has at least a cut-vertex.
As a result, $H$ has at least two blocks containing a cycle.
Let $D$ be a block of $H$ that contains a cycle.
If $D$ is a cycle, then \cite[Lemma~4]{Sed21} implies that we can find a mixed metric generator $X_D$ containing a cut-vertex.
Suppose that $c(D)=2$. Then $D$ is a $\Theta$-graph by Lemma~\ref{lem-1}. If $D$ is unbalanced, then \cite[Theorem 10]{Sed} manes that $\dm(D)=2c(D)-1$;
Otherwise, by Lemma~\ref{lem-2} we may choose
a mixed metric generator $X_D$ containing a cut-vertex.
If $c(D)\ge 3$, since $D$ is $2$-connected and $c(D)<n$,
we have $\dm(D)\le 2c(D)-1$ by the induction hypothesis, which implies that we can take a mixed metric generator $X_D$ of $D$ that has size $2c(D)-1$.

\medskip
\noindent {\bf Case 2.1.} $H$ has a block $D'$ that contains a cycle but is not itself a cycle.
\medskip

In view of Lemma~\ref{lem-3}, we have that
\begin{equation}\label{eq-9}
\begin{aligned}
 \dm(H) & \le \Big|\bigcup_{D\in \mathcal{F}}(X_D\setminus C_D)\Big|\\
   & \le \Big|\bigcup_{D\in \mathcal{F}\setminus\{D'\}}(X_D\setminus C_D)\Big|+
   \big| X_{D'}\setminus C_{D'}\big|\\
   & \le \sum_{D\in \mathcal{F}\setminus\{D'\}}2c(D)+
   2c(D')-1\\
   & = 2c(H)-1.
\end{aligned}
\end{equation}
Now, from \cite[Lemma~2.4]{C}, it follows that
$$\dm(G)\le \dm(H)+2\le 2c(H)+1=2c(G)-1,$$
as desired.

\medskip
\noindent {\bf Case 2.2.} Every $D\in \mathcal{F}$ is a cycle.
\medskip

In this case we have that $H$ is a cactus graph.
By \cite[Corollary 10]{Sed21}, we conclude that $\dm(H)\le 2c(H)$ and the equality is attained
if and only if every cycle in $H$ has exactly one root vertex.
If $\dm(H)\le 2c(H)-1$, then the desired result follows
from \cite[Lemma~2.4]{C}. Suppose now that $\dm(H)=2c(H)$. Since $G$ is $2$-connected, we deduce that $H$ is a
cactus graph with exactly two cycles. As a result, we have that $c(H)=2$, which implies that $c(G)=3$, and the required result follows.
$\qed$

\section{An affirmative answer to Conjecture~\ref{conj-2}}

We are now ready to prove Conjecture~\ref{conj-2}.

\medskip

{\noindent \em Proof of Conjecture $\ref{conj-2}$.}
By \cite[Corollary 8]{Sed},
\eqref{th-eq-1} holds for a
balanced $\Theta$-graph $G$.
Moreover, it follows from \cite[Theorem 9]{Sed21} that
\eqref{th-eq-1} holds for a cactus graph in which every cycle has precisely one vertex of degree at least $3$.
Thus, the sufficiency follows.

We next prove the necessity.
Suppose that \eqref{th-eq-1} holds for graph $G$.
If $c(G)=0$, then $G$ is a tree, which implies that $G$ is a cactus graph in which every cycle has precisely one vertex of degree at least $3$ (note that a tree has no cycles), as desired. In the following, we always assume that $c(G)\ge 1$.

\medskip
\noindent {\bf Case 1.} $G$ has no cut-vertices.
\medskip

If $c(G)=1$, then $G$ is a cycle, which implies that
\eqref{th-eq-1} is not valid by
\cite[Theorem~3.8 and Proposition~4.1]{Kek}.
Suppose that $c(G)=2$. Then by Lemma~\ref{lem-1}, $G$ is a $\Theta$-graph. Furthermore, if $G$ is an unbalanced $\Theta$-graph, then by \cite[Theorem 10]{Sed} we have $\dm(G)=3$, which implies \eqref{th-eq-1} does not hold.
Otherwise, $G$ is a balanced $\Theta$-graph, as desired.

Now, suppose that $c(G)>2$. Then Lemma~\ref{lem-4} means that
$$\dm(G)\le 2c(G)-1<2c(G)\le \ell(G)+2c(G),$$
which is impossible.

\medskip
\noindent {\bf Case 2.} $G$ has at least a cut-vertex.
\medskip

If $G$ has a block that contains a cycle but is not itself a cycle, then by Lemma~\ref{lem-3}, it is similar to \eqref{eq-9}, we deduce that $\dm(G)\le 2c(G)-1$, which is impossible. It follows that every block of $G$ that contains a cycle is itself a cycle. As a result, $G$ is a cactus graph.
Moreover, by \cite[Theorem 9]{Sed21}, we see that $G$ is a cactus graph in which every cycle has precisely one vertex of degree at least $3$, as desired.
$\qed$

\section*{Conflict of interest}
The authors state no conflict of interest.

\end{CJK*}

\end{document}